\documentclass{amsart}
\usepackage{amssymb,latexsym}
\usepackage{enumerate} 

\newtheorem{theorem}    {Theorem}[section] 

\newtheorem{lemma}      [theorem]{Lemma}
\newtheorem{proposition}[theorem]{Proposition}

\theoremstyle{definition}

\newtheorem{remark}     [theorem]{Remark}

\newtheorem*{rem}{Remark}

\theoremstyle{remark}
\newtheorem*{assumption}         {Assumption}

\numberwithin{equation}{section}

\newcommand\C{\mathbb C}
\newcommand\F{\mathbb F}
\newcommand\Q{\mathbb Q}
\newcommand\Z{\mathbb Z}
\newcommand{\Aut}{\operatorname{Aut}}
\newcommand{\disc}{\operatorname{disc}}
\newcommand{\Gal}{\operatorname{Gal}}
\newcommand{\Ker}{\operatorname{Ker}}
\renewcommand{\Re}{\operatorname{Re}}
\newcommand{\trace}{\operatorname{trace}}

\begin{document}

\title[Supersingular distribution on average]{Supersingular distribution on average\\ for congruence classes of primes}

\author[N. Walji]{Nahid Walji}
\address{Department of Mathematics MC 253-37, California Institute of Technology\\
Pasadena, California 91125, U.S.A.}
\email{nahid@caltech.edu}

\date{}

\begin{abstract}
A general question of interest is to know, for elliptic curves over a number field, the distribution of supersingular primes of degree one. We begin by modifying the Lang-Trotter heuristic to address the case of an abelian extension, then we show that it holds on average (up to a constant) for a family of elliptic curves by using ideas of David-Pappalardi. 
\end{abstract}

\maketitle

\section*{Introduction}
Given an elliptic curve $E$ defined over a number field $F$ and a prime $\mathfrak{p}$ of good reduction, let $a_{\mathfrak{p}}= N\mathfrak{p} + 1-|E(\mathbb{F}_{\mathfrak{p}})|$. One says that $\mathfrak{p}$ is \textit{supersingular} for $E$ if $a_{\mathfrak{p}} \equiv 0 \pmod{\mathfrak{p}}$. Given Hasse's bound of $2\sqrt{N \mathfrak{p}}$ on $|a_{\mathfrak{p}}|$, this occurs for degree one primes $\mathfrak{p}$ with $N\mathfrak{p}>3$ iff $a_{\mathfrak{p}}=0$.

In 1976, Lang and Trotter~\cite{LT76} conjectured the distribution of supersingular primes of norms $\leq x$ for a non-CM elliptic curve $E$ over $\Q$ to be asymptotic to
\begin{align*}
c_E \frac{\sqrt x}{\log x}
\end{align*}
as $x \rightarrow \infty$, for some positive constant $c_E$.

In 1981, Serre~\cite{Se81} proved that the density of supersingular primes is zero, which brings up the question of whether such primes are infinite in number for any $E/F$. This was answered for $F=\Q$ by N.Elkies in 1987~\cite{El87}, when he demonstrated that there is an infinite number of supersingular primes for any elliptic curve $E/\Q$. In 1989, Elkies~\cite{El89} extended his result to an elliptic curve $E$ over any number field with a real embedding.

No such result is known in general for other fields $K$, even in the case of imaginary quadratic fields. Some examples are provided by Elkies and Jao, in the cases where the absolute norm of $j(E)-1728$ has a prime factor congruent to $1 \pmod 4$ which has odd exponent~\cite{El87} and for elliptic curves parameterized by $X_0(p)/w_p$ for certain small odd primes $p$~\cite{J05} (which includes some elliptic curves with imaginary quadratic $j$-invariant). It does not seem that the result on $\Q$ can be used to deduce it for $K$, though this may be possible if one can show that infinitely many of the supersingular primes over $\Q$ have a degree one divisor over $K$.

This is certainly not always true, as in the case of an elliptic curve $E$ with a torsion subgroup $E(\Q)_{\text{tors}}$ of order 4 and the imaginary quadratic field $\Q(i)$. Recall the theorem on torsion injection which states that for sufficiently large primes $p$, we have $E(\Q)_{\text{tors}} \hookrightarrow E(\F_p)$. In our case this implies $a_p \equiv p+1 - |E(\F_p)| \equiv p+1 \pmod 4$ and so $p$ supersingular $\Rightarrow p \equiv 3 \pmod 4$ for
sufficiently large primes $p$. Thus $E$ does not have an infinite number of supersingular primes that split in $\Q(i)$.

In 1996, Fouvry and Murty~\cite{FM96} demonstrated that the Lang-Trotter asymptotic for supersingular primes held on average for a family of elliptic curves over $\Q$. David and Pappalardi~\cite{DP99} later established asymptotic expressions on average for any given trace of Frobenius, and also investigated the average distribution of supersingular inert primes of $\Q(i)$~\cite{DP04}. We will use their techniques to establish a result for a congruence class of primes, averaging over a family of elliptic curves. Such techniques were used in a similar manner by Kevin James in~\cite{Ja05}, however the aim of that paper is quite different from our goal here.

The surprise of the findings is that whilst the prediction is true up to a constant, the nature of the specific constant is not what is expected.

Fix $L$ to be an abelian extension and let $\pi_0(L,E_{a,b},x)$ be the number of rational primes less than $x$ that split (totally) in $L$ and are supersingular for $E_{a,b}$, where $E_{a,b}$ is the elliptic curve represented by the equation $Y^2=X^3+aX+b$. We will show, for $A,B > x^{1/2+\epsilon}$ and $AB > x^{3/2+\epsilon}$, that
\begin{align*}
\frac{1}{4AB} \sum_{|a|\leq A} \sum_{| b|\leq B} \pi_0(L,E_{a,b},x)
\sim C_L \frac{\sqrt{x}}{\log x} 
\end{align*}
as $x \rightarrow \infty$, where $C_L$ is an explicit positive constant that is addressed in section~\ref{s4}.
For example, if we set $L=\Q(\sqrt{-3})$, then 
\begin{align*}
\frac{1}{4AB} \sum_{|a|\leq A} \sum_{|b|\leq B} \pi_0 \left(\Q\left(\sqrt{-3}\right), E_{a,b}, x \right)
\sim \frac{\pi}{9} \frac{\sqrt{x}}{\log x}.
\end{align*}
\begin{rem}
Note that the constant above of $\pi/9$ is less than half of $C_\Q=\pi/3$, suggesting a slight \textit{bias} against the occurrence of supersingular primes that split in $\Q(\sqrt{-3})$. In section~\ref{s5}, we will discuss the existence of this bias in the averaging result. 
\end{rem}

We plan to investigate elsewhere the reasons for this bias, in part through the use of Galois representations.

We briefly describe the structure of the paper. In section~\ref{s1} we construct the heuristic, using a similar approach to that of Lang-Trotter. We then turn to averaging over a family of elliptic curves in section~\ref{s3}, expressing the average in terms of sums of Hurwitz numbers. We interpret these in section~\ref{s4} using $L$-functions and obtain an asymptotic expression. In section~\ref{s5} we discuss the results when applied to certain examples of number fields, and lastly in section~\ref{s6} we explain how to refine the averaging, using an idea of Fouvry-Murty.

\section{The heuristic} \label{s1}
\subsection{Background on Sato-Tate}
The heuristic will rely on a variant of the Sato-Tate conjecture. Given a non-CM elliptic curve $E$ defined over a number field $L$ of degree $n$ over $\Q$, let $\Sigma_L =\Sigma_{L,E}$ be the set of finite places of $L$ at which $E$ has good reduction. For $v \in \Sigma_L$ and $(l,\mathfrak{p}_v)=1$ (where $\mathfrak{p}_v$ is the prime associated to the place $v$), we have the related Frobenius conjugacy class in $\Aut(T_l(E))$. From page I-25 of~\cite{Se98} we have:
\begin{lemma}
 The eigenvalues of this Frobenius conjugacy class, when embedded into $\C$, are $\pi_v$ and $\overline{\pi_v}$, where $\pi_v = (Nv)^{1/2} e^{i\phi_v}$ for some $\phi_v \in [0,\pi]$.
 \end{lemma}
We say that the \textit{Sato-Tate conjecture} holds for $\Sigma \subset \Sigma_L$ if the angles $\phi_v$, $v \in \Sigma$ of the Frobenius elements are equidistributed with respect to the measure $\frac{2}{\pi} \sin^2\phi \,d \phi$.
\begin{proposition}
 Consider the hypotheses:
\begin{enumerate}[\upshape (i)]
\item $\prod_{v \in \Sigma} (1-(Nv)^{-s})^{-1}$ converges for $\Re(s)>1$ and extends to a meromorphic function on $\Re(s) \geq 1$ with the only zero or pole in the region being a simple pole at $s=1$, \label{item1}
\item $L_m(s) = \prod_{v \in \Sigma} \prod_{a=0}^{m} (1-e^{i(m-2a)\phi_v} (Nv)^{-s})^{-1}$ extends to a non-zero holomorphic function on $\Re(s) \geq 1$ for all $m \geq 1$. \label{item2}
\end{enumerate}
If both~(\ref{item1}) and~(\ref{item2}) are true, then the Sato-Tate conjecture holds for $\Sigma$.
\end{proposition} 
\begin{proof}
Pages I-22 - I-26 of~\cite{Se98}.
\end{proof}

Given an abelian extension $L/\Q$ and assuming the Sato-Tate conjecture for $\Sigma_{L}$, it is possible to derive the Sato-Tate conjecture for the set of degree one primes of $L$ through the use of some straightforward complex analysis.

\subsection{Setup} \label{sbsec:setup}
Let $G=\Gal(\overline{L}/L)$. Given a non-CM elliptic curve $E/L$, we have a representation $\rho: G \rightarrow \prod_{l}{\rm GL}_2(\Z_l)$. Now fix a positive integer $M$ and reduce $\prod {\rm GL}_2(\mathbb{Z}_l)$
modulo $M$ so as to obtain the representation $\rho_{M}: G \rightarrow {\rm GL}_2(\mathbb{Z}/M\mathbb{Z})$. 

We choose an $M$ that \textit{splits} and \textit{stabilizes} the representation, which from Serre~\cite{Se72} we know always exists for a non-CM elliptic curve. For good primes $\mathfrak{p}$ of $E$ we have the trace $t_\mathfrak{p}$ of the Frobenius class $\sigma_\mathfrak{p}$, and it is the distribution of these $t_\mathfrak{p}$ for $\mathfrak{p} \in R$ that we consider, for a given set $R$ of primes in $L$.
\begin{assumption}
 $t_\mathfrak{p}$ can be considered as a random variable, independent of other $t_{\mathfrak{p}'}$, and its asymptotic behaviour is consistent with the Sato-Tate conjecture and the Tchebotarev density theorem.
 \end{assumption} 
Let $G(M) = G / \Ker \rho_{M}$, $G(M)_t=\{\sigma \in G(M) \mid \trace(\sigma) \equiv t \pmod M \}$ and $F_M(t) = M |G(M)_t|/|G(M)|$. Note that $\frac{1}{M} \sum_{t \pmod M} F_M(t) = 1$, i.e. the average value of $F_M(t)$ is 1.
Define $\xi(t,\mathfrak{p}) = t / 2\sqrt{N\mathfrak{p}}$ so that, by Hasse's inequality, $\xi \in [-1,1]$ and
thus we can write $\xi(t,\mathfrak{p})= \cos \theta (t,\mathfrak{p})$, with $\theta(t,\mathfrak{p}) \in [0,\pi]$. By the Sato-Tate result, we have a density function of $\frac{2}{\pi} \sin^2 \theta$, and this gives us a density function for $\xi$ of $g(\xi) = \frac{2}{\pi} \sqrt{1-\xi^2}$ on $[-1,1]$.

\subsection{Approximation model} \label{sbsec:am}
We now construct an approximation model that is slightly different to that of Lang-Trotter, where we also account for the torsion of the elliptic curve. We then determine the asymptotic behaviour.

Let $k$ be the order of the torsion in the Mordell-Weil group of the elliptic curve. One knows that for sufficiently large primes $\mathfrak{p}$ we have $|E(\F_\mathfrak{p})| \equiv 0 \pmod k$. 
Let $f_M(t,\mathfrak{p},k)= \text{prob}_M\{a_\mathfrak{p} \equiv t \pmod M\text{ and } a_\mathfrak{p} \equiv N\mathfrak{p}+1 \pmod k\}$, and define $h(t,\mathfrak{p},k)$ to be equal to $k$ if $t \equiv N\mathfrak{p}+1 \pmod k$ and $0$ otherwise.
\begin{remark}
Note that we require $\sum_{t}f_M(t,\mathfrak{p},k)=1$.
 \end{remark} 
The main assumption of the heuristic is that we have
\begin{align*}
f_M(t,\mathfrak{p},k) = c_\mathfrak{p} \cdot g(\xi(t,\mathfrak{p})) \cdot F_M(t) \cdot h(t,\mathfrak{p},k).
\end{align*}
\begin{remark}
 The $F_M(t)$ is for consistency with the Tchebotarev density theorem, $g(\xi(t,\mathfrak{p}))$ is to ensure
 compatibility with the Sato-Tate result (via the law of large numbers), $c_\mathfrak{p}$ is chosen so that $\sum_t f_M(t,\mathfrak{p},k)=1$, and lastly $h(t,\mathfrak{p},k)$ accounts for the congruence obstruction from torsion injection.
\end{remark}
 We now address the asymptotic behaviour of $c_\mathfrak{p}$.
\begin{lemma}
If $\sum_t f_M(t,\mathfrak{p},k)=1$, then
\begin{align*}
c_\mathfrak{p} \sim
\frac{1}{2\sqrt{N\mathfrak{p}}}    
\end{align*}
as $N\mathfrak{p} \rightarrow \infty$.
\end{lemma} 
\begin{proof}
 We can prove this in a similar manner to the corresponding case in~\cite{LT76}, with the additional observation that when incompatible congruence conditions arise the proof does still follow, due to the fact that for the $t_0$ in question we will have $F(t_0)=0$.
\end{proof} 
Now let $R$ be the set of degree one primes of an abelian extension $L$. If we have incompatibility between the congruence arising from torsion injection and the congruences satisfied by the norms of the elements of $R$, then the set $Q = \{\mathfrak{p} \in R \mid N \mathfrak{p} + 1 \equiv 0 \pmod k \}$ may only have a finite number of elements (as in the example mentioned in the introduction). Thus from here on we only consider the cases where this does not occur (i.e. when $(k,\disc L)=1$).
In this case the set $Q'$ of rational primes lying under those primes in $Q$ will have density $1/nk$, where we recall that $n$ is the degree of the field $L$ and $k$ is the order of the rational torsion group of the elliptic curve. The rest of the analysis of the model follows as in~\cite{LT76}, and we obtain:
\begin{theorem} 
Let $N_{t_0,R}(x)$ be the number of primes $\mathfrak{p} \in R$ such that $N \mathfrak{p} \leq x$ and $t_\mathfrak{p}=t_0$. Then
\begin{align*}
N_{t_0,R}(x) \sim
\frac{2}{\pi} F_M(t_0) \left(\prod_{(l,M)=1} F_l(t_0)\right) k
\sum_{\substack{N\mathfrak{p} \leq x \\ \mathfrak{p} \in R \\ N \mathfrak{p} + 1 \equiv 0 (k)}} \frac{1}{2\sqrt{N\mathfrak{p}}}, 
\end{align*}
where $l$ is a rational prime. When $t_0=0$, we have (given that $M$ splits $\rho$)
\begin{align*}
N_{0,R}(x) &\sim
\frac{2}{\pi} F_M(0) \zeta(2)
\left(\prod_{l|M}(1-1/l^2)\right) k
\sum_{\substack{N\mathfrak{p} \leq x \\ \mathfrak{p} \in R \\ N \mathfrak{p} + 1 \equiv 0 (k)}}
\frac{1}{2\sqrt{N\mathfrak{p}}}.
\end{align*}
\end{theorem}
We consider this in the context of rational primes. Let $R'$ be the set of rational primes that lie under the primes in $R$; this has density $1/n$. Thus
\begin{align*}
N_{0,R'}(x) \sim \frac{2}{\pi} F_M(0) \zeta(2)
\left(\prod_{l|M}(1-1/l^2)\right)
\frac{1}{n} \frac{\sqrt{x}}{\log x}.
\end{align*}
\begin{remark}
 Comparing the above with the Lang-Trotter conjecture of 
\begin{align*}
N_{0}(x) \sim 
\frac{2}{\pi} F_M(0) \zeta(2)
\left(\prod_{l|M}(1-1/l^2) \right)
\frac{\sqrt{x}}{\log x},
\end{align*}
 we see that $R'$ has an asymptotic proportion of supersingular primes of $1/n$ (when there is no congruence obstruction). However, later in this paper we shall see that on average there is in fact a bias in the proportion of supersingular primes, reflected in the value of the constant, which is contrary to the expectations of the heuristic.
\end{remark}

\section{Lifting supersingular elliptic curves} \label{s3}
Define $E_{a,b}$ to be the elliptic curve with affine model $Y^2=X^3+aX+b$, with $a,b \in \Z$, and given a set of rational primes $P$, let $\pi_0(P,E_{a,b},x)$ be the number of supersingular primes for $E_{a,b}$ that are elements of $P$ and less than or equal to $x$.

When $P$ is determined by a congruence condition $p \equiv c\pmod m$ with $(c,m)=1$, then we will prove
\begin{theorem}
For $A,B > x^{1/2+\epsilon}, AB > x^{3/2+\epsilon}$ we have
\begin{align*}
\frac{1}{4AB} \sum_{|a| \leq A} \sum_{|b| \leq B} \pi_0 (P,E_{a,b},x)
\sim C_P \frac{\sqrt{x}}{\log x}
\end{align*}
as $x \rightarrow \infty$, for an explicitly determined positive constant $C_P$ that depends only on $P$.
\end{theorem}

By Deuring, the set of supersingular primes for a CM elliptic curve over $\Q$ has density $1/2$. Thus we must check whether the constant in the asymptotic is affected by the contribution from CM curves.

There are exactly thirteen isomorphism classes of CM elliptic curves over $\Q$. Two of these classes are of the form $E_{a,0}$ and $E_{0,b}$, where $a$ and $b$ can be any non-zero integers, and the other eleven classes are of the form $E_{c_i t^2, d_i t^3}$, where $t$ is any non-zero integer and $(c_i,d_i)$ is, depending on its index $i$, one of eleven distinct pairs of integers.

Therefore
\begin{align*}
& \sum_{|a|\leq A} \sum_{\substack{|b|\leq B \\ E_{a,b} \text{ is CM}}} \pi_0 (P,E_{a,b},x) \\
& \ll  \frac{1}{2}\frac{x}{\log x} 2A 
+ \frac{1}{2}\frac{x}{\log x} 2B
+ \frac{x}{\log x} \cdot \min (A^{1/2}, B^{1/3}) \\
&\ll \frac{x}{\log x} \cdot \max (A,B)
\end{align*}
which we note will not affect the asymptotic expression under the conditions $A,B >  x^{1/2+\epsilon}, AB > x^{3/2+\epsilon}$. We will later choose some examples of abelian extensions and use congruence conditions to determine the set of rational primes lying under the degree one primes of the extension. This will enable us to obtain explicit asymptotic expressions.

Given a prime $p > 3$ and an integer $r \in (-2 \sqrt{p},  2 \sqrt{p})$, the number of isomorphism classes of elliptic curves over $\F_p$ with $a_{p}=r$ is equal to the Hurwitz number 
\begin{align} \label{alpha}
 H(r^2-4p) = 2 \sum_{\substack{f^2|(r^2-4p) \\ d \equiv 0,1(4)}} \frac{h(d)}{w(d)}
\end{align}
where $h(d)$ is the class number of the order $\Z[(d+\sqrt{d})/2]$ of discriminant $d$, $w(d)$ is the number of units in this order, $f$ is an integer and $d=(r^2-4p)/f^2$. We define 
\begin{align*}
& \delta_f(x) = \left\{3 < p \leq x \middle| r^2 - 4p \equiv 0 \pod{f^2} \text{ and } \frac{r^2 - 4p}{f^2} \equiv 0 \text{ or } 1 \pod 4 \right\}.
\end{align*}
Note that the $r=0$ case only allows for two values of $f$, either 1 or 2. Thus $d=-4p$ or $-p$.

From an argument of Fouvry-Murty using exponential sums ~\cite{FM96}, for any $\alpha,\beta \in \{1,\ldots,p-1\}$, we have
\begin{align*}
|\{(a,b) \mid |a|\leq A, |b|\leq B, \text{ and } E_{a,b}\pmod p \text{ is isomorphic to }E_{\alpha, \beta} \}| \\
= \frac{p-1}{2} \cdot \frac{4AB}{p^2}
 + O\Bigg(\sqrt{p} \log^2 p \left(1+\frac{A}{p} + \frac{B}{p}\right) \Bigg) .
\end{align*}
We will sum this over the $H(-4p)-O(1)$ isomorphism classes that have $\alpha \beta \not\equiv 0 \pmod p$. In order to bound the error terms that will arise, we note that $H(-4p) = h(-p) + h(-4p)$, and we recall the Dirichlet class number formula
\begin{align*}
h(d)= \frac{w(d)\sqrt{|d|}}{2 \pi} L(1,\chi_{d})
\end{align*}
when $d \equiv 0\text{ or } 1 \pmod 4$, with $w(d)=6,4,2$ for $d=3,4,\text{ or }\geq 7$, respectively. The Polya-Vinogradov inequality implies that $L(1,\chi_d) \ll \log |d|$ and thus
\begin{align}
\label{wo} H(-4p) \ll \sqrt{p} \log p.
\end{align}
Therefore, we obtain
\begin{align}
\label{ln} \frac{1}{4AB} \sum_{\substack{|a| \leq A \\ |b| \leq B}} \pi_0 (P,E_{a,b},x)
&= \frac{1}{2} \sum_{3 < p \leq x,  p \in P} \frac{H(-4p)}{p} \\  
& + O\Bigg(\frac{x^2 \log^2 x}{AB} + \log \log x + \left(\frac{1}{A} + \frac{1}{B} \right) x \log^2 x \Bigg). \notag
\end{align}

\section{Sums of L-functions} \label{s4}
\subsection{The asymptotic expression}
Using the Dirichlet class number formula in conjunction with~(\ref{alpha}), we obtain:
\begin{align} \label{cross}
\frac{1}{2} \sum_{3 < p \leq x, p \in P}\frac{H(-4p)}{p} = \frac{1}{\pi} \sum_{f=1,2} \frac{1}{f} \sum_{p \in \delta_f (x), p \in P} \frac{L(1,\chi_{d})}{\sqrt{p}}. 
\end{align}
Now we note that the double sum on the right-hand side can be considered as a Stieltjes integral, which gives us
\begin{align} \label{cross star}
&\frac{1}{\sqrt{x}\log x} \sum_{f=1,2} \frac{1}{f} \sum_{\substack{p \in \delta_f (x) \\ p \in P}} L(1,\chi_d) \log p \\
&- \int_{2}^{x}\sum_{f=1,2} \frac{1}{f} \sum_{\substack{p \in \delta_f (t) \\ p \in P}} L(1,\chi_d) \log p \cdot \frac{d}{dt} \left(\frac{1}{\sqrt{t}\log t}\right)\, dt . \notag 
\end{align}
At this point we require:
\begin{theorem}
\begin{align}
 \label{star}
\sum_{f=1,2}\frac{1}{f} \sum_{\substack {p \in \delta_f (x)\\ p \in P}} L(1,\chi_{d}) \log p = K_P x + O\left(\frac{x}{\log^\gamma x}\right)
\end{align}
where $K_P$ is a constant depending only on $P$, and $\gamma$ is some positive constant.
\end{theorem}

\begin{proof}
Since the first part of the proof follows from a straightforward adjustment of the arguments of David-Pappalardi, we omit it here.
We obtain two terms, one for each value of $f$ in the sum above.

 The first term, corresponding to the $f=1$ case, is
\begin{align} \label{20}
& x \sum_{\text{odd }n \leq U}\frac{1}{n} \sum_{\substack{b(n)^* \\ b \equiv -c(k)}} \left(\frac{b}{n}\right) \frac{1}{\varphi (nm/k)} + O\left( \frac{x }{ \log ^\gamma x}\right)
\end{align}
where $k=k(n)$ is the greatest common factor of $n$ and $m$, and $U=\sqrt{x}\log^{\gamma+1}x$ for some $\gamma > 0$.
\begin{lemma}
\label{lem1} For $k=(n,m)$, we have
\begin{align*}
& \varphi \left(\frac{nm}{k}\right) = \frac{\varphi(n)\varphi(m)}{\varphi(k)}.
\end{align*} 
\end{lemma}
\begin{proof}
This can be proved by considering the prime decompositions of $n$, $m$ and $k$.
\end{proof}
We apply this to~(\ref{20}) to get
\begin{align}
 \frac{x}{\varphi(m)} \sum_{\text{odd }n \leq U} \frac{1}{n \varphi(n) / \varphi(k)} \sum_{\substack{b(n)^* \\ b \equiv -c(k)}} \left(\frac{b}{n}\right). \label{ku}
\end{align}
Consider the inner sum in the equation above. Since $k|n$, write $n=lk$ and note that $k|l$ only if $k \nmid \frac{m}{k}$. We have, given the conditions of the sum above, that $\left(\frac{b}{n}\right)
 = \left(\frac{b}{l}\right) \left(\frac{b}{k}\right)
 = \left(\frac{b}{l}\right) \left(\frac{-c}{k} \right)$ and thus we can re-express the inner sum as
\begin{align} \label{10''}
& \left(\frac{-c}{k}\right) \sum_{\substack{b(n)^* \\ b \equiv -c(k)}} \left(\frac{b}{l}\right).
\end{align}

Decompose $l$ into its prime factors as $(p_1^{\alpha_1} \dots p_t^{\alpha_t}) \cdot (p_{t+1}^{\alpha_{t+1}} \dots p_q^{\alpha_q})$, where $p_1, \dots ,p_t|k$ and $p_{t+1}, \dots, p_q \nmid k$ and let $l_k=p_1^{\alpha_1} \dots p_t^{\alpha_t}$, $l_s = p_{t+1}^{\alpha_{t+1}} \dots p_q^{\alpha_q}$. 
Then~(\ref{10''}) becomes
\begin{align*} 
& \left(\frac{-c}{k}\right) \left(\frac{-c}{p_{1}^{\alpha_{1}} \dots p_t^{\alpha_t}}\right)
\sum_{\substack{b(n)^* \\ b \equiv -c(k)}} \left(\frac{b}{l_s}\right)
= \left(\frac{-c}{kl_k}\right) \sum_{\substack{b(n)^* \\ b \equiv -c(k)}} \left(\frac{b}{l_s}\right).
\end{align*}
\begin{lemma} \label{bu}
If $l_s$ is a square, then
\begin{align} \label{wu}
& \sum_{\substack{b(n)^* \\ b \equiv -c(k)}} \left(\frac{b}{l_s}\right)
= \frac{\varphi(n)}{\varphi(k)},
\end{align}
otherwise the sum takes the value zero.
\end{lemma}
\begin{proof}
When $l_s$ is a square, we have 
\begin{align*}
\sum_{\substack{b(n)^* \\ b \equiv -c(k)}} \left(\frac{b}{l_s}\right)
= & \sum_{\substack{b(n)^* \\ b \equiv -c(k)}} 1
= \frac{\varphi(n)}{\varphi(k)}.
\end{align*}
On the other hand, if $l_s$ is not a square, write $l_s = q_1 \dots q_w \cdot h^2$, where the $q_i$ are distinct odd primes ($n$ is odd so $l_s$ is also odd). Then 
\begin{align*}
\sum_{\substack{b(n)^* \\ b \equiv -c(k)}} \left(\frac{b}{l_s}\right)
=& \sum_{\substack{b(n)^* \\ b \equiv -c(k)}} \left(\frac{b}{q_1}\right) \dots \left(\frac{b}{q_w}\right)
 \cdot \left(\frac{b}{h}\right)^2 \\
=& \frac{\varphi(n)}{\varphi(k) \varphi(q_1 \dots q_w)}
 \sum_{\substack{b(q_1 \dots q_w)^*}} \left(\frac{b}{q_1}\right) \dots \left(\frac{b}{q_w}\right)
\end{align*}
which, by the Chinese remainder theorem, gives us 
\begin{align*}
& \frac{\varphi(n)}{\varphi(k) \varphi(q_1 \dots q_w)}
 \left(\sum_{\substack{b(q_1)^*}} \left(\frac{b}{q_1}\right) \right)
 \dots \left(\sum_{\substack{b(q_w)^*}} \left(\frac{b}{q_w}\right) \right) 
\end{align*}
and since $\sum_{\substack{b(q_i)^*}} \left(\frac{b}{q_i}\right) = 0$, the left-hand side of~(\ref{wu}) is zero.
\end{proof}

Applying all this to~(\ref{ku}) we have
\begin{align*}
 \frac{x}{\varphi(m)} \sum_{\text{odd }n \leq U}
 \frac{\sum_{\substack{b(n)^*,b \equiv -c(k)}} \left(\frac{b}{n}\right)} {n \varphi(n) / \varphi(k)} 
=& \frac{x}{\varphi(m)} \sum_{\substack{\text{odd }n \leq U \\ l_s \text{ a square}}}
 \left(\frac{-c}{kl_k}\right) \frac{1}{n}.
\end{align*}
We want to split up this sum according to the different values of $k$. In order to do so, we first need to extend the range of summation:
\begin{align*}
\frac{x}{\varphi(m)} \sum_{\substack{\text{odd }n \leq U \\ l_s \text{ a square}}} \left(\frac{-c}{kl_k}\right) \frac{1}{n}
=& \frac{x}{\varphi(m)} \sum_{\substack{\text{odd }n \geq 1 \\ l_s \text{ a square}}} \left(\frac{-c}{k l_k}\right) \frac{1}{n}
+ O\left(\frac{x}{\varphi (m)} \sum_{\substack{\text{odd }n > U \\ l_s \text{ a square}}} \frac{1}{n}\right) 
\end{align*}
In the context of the error term above we note
\begin{align*}
\frac{x}{\varphi (m)} \sum_{\substack{\text{odd }n > U \\ l_s \text{ a square}}} \frac{1}{n}
\ll& x\sum_{\substack{\text{odd square } n > U}} \frac{1}{n} \\
\ll& \frac{\sqrt{x}}{\log^{\gamma+1}x}.
\end{align*}
Now we can split the main term into a double sum 
\begin{align}
& \frac{x}{\varphi (m)} \sum_{\substack{\text{odd }n \geq 1 \\ l_s \text{ a square}}} \left(\frac{-c}{k l_k}\right) \frac{1}{n}
 = \frac{x}{\varphi (m)} \sum_{\text{odd }k|m} \sum_{\substack{\text{odd }l \geq 1 \\ l_s \text{ a square} \\ (l, m/k)=1}} \left(\frac{-c}{k l_k}\right) \frac{1}{lk} \label{cross''}
\end{align}
where the coprime condition on $l$ arises from the fact that we have $k=(n,m) \Leftrightarrow (n/k,m/k)= 1 \Leftrightarrow (l,m/k)=1$. Note that the first sum on the right-hand side is over all odd divisors of $m$, including $1$ and possibly $m$. 

In order to interpret this, we consider the possible values of the divisors $l_k$ and $l_s$ of $l$ (recall that $l=l_k l_s$). For $l_k$, we note that primes $p$ such that $p|k$ and $p \nmid m/k$ are exactly those primes that might divide $l_k$. Therefore, we have
\begin{align*}
\sum_{l_k} \frac{\left(\frac{-c}{l_k}\right)}{l_k}
=& \prod_{\substack{p|k \\ p \nmid m/k}} \left(  \sum_{w = 1}^{\infty} \frac{\left(\frac{-c}{p}\right)^ w}{p^w} \right) \\
=& \prod_{\substack{p|k \\ p \nmid m/k}} \left(1-\frac{\left(\frac{-c}{p}\right)}{p}\right)^{-1}.
\end{align*}
For $l_s$, recall that it is the largest divisor of $l$ such that $(l_s,k)=1$. We also have that $(l,m/k)=1$ in the second sum of the right-hand side of~(\ref{cross''}), and so $(l_s,m/k)=1$. Thus $(l_s,k \cdot m/k)=1 \Rightarrow (l_s,m)=1$.
So $l_s$ is any square positive integer that satisfies $(l_s,m)=1$ and therefore we can write 
\begin{align*}
& \sum_{\substack{\text{odd }l \geq 1 \\ (l_s,m)=1}} \frac{1}{l_s}
= \zeta(2) \prod_{p|2m} \left(1-\frac{1}{p^2}\right).
\end{align*}

Putting all this together, we conclude that~(\ref{20}) is equal to
\begin{align*}
 \frac{x\zeta(2)}{\varphi(m)} \left(\prod_{p|2m} \left(1-\frac{1}{p^2}\right) \right) \sum_{\text{odd }k|m} \frac{\bigl(\frac{-c}{k}\bigr)}{k} \prod_{\substack{p|k \\ p \nmid m/k}} \left( 1-\frac{\left(\frac{-c}{p} \right)}{p} \right)^{-1}
+ O\left(\frac{x}{\log^\gamma x}\right).
\end{align*}

Now for our second term, corresponding to the $f=2$ case, we have
\begin{align}\label{beta}
x  \sum_{n \leq U} \frac{1}{n} \sum_{\substack{b(4n)^* \\ b \equiv -c'(4k)}} \left(\frac{b}{n}\right) \frac{1}{\varphi(4nm/k)} + O\left( \frac{x}{\log^\gamma} \right),
\end{align}
where $c'$ is an integer such that $ c' \equiv  c \pmod k$ and $ c' \equiv 3 \pmod 4$.
A similar approach as before to the evaluation of this expression gives us, in the case of $m$ odd,
\begin{align*}
 \frac{x\zeta(2)}{4 \varphi(m)} \left(\prod_{p|m} \left(1-\frac{1}{p^2}\right)\right) \sum_{k|m} \frac{\left(\frac{-c}{k}\right)}{k} \prod_{\substack{p|k \\ p \nmid m/k}} \left( 1-\frac{\left(\frac{-c}{p}\right)}{p} \right)^{-1}
+ O\left(\frac{x}{\log^\gamma x}\right),
\end{align*}
and in the case of $m$ divisible by 4,
\begin{align*}
 \frac{x\zeta(2)}{2 \varphi(m)} \left(\prod_{p|m/4} \left(1-\frac{1}{p^2}\right) \right) \sum_{k|m/4} \frac{\left(\frac{-c}{k}\right)}{k} \prod_{\substack{p|k \\ p \nmid m/4k}} \left( 1-\frac{\left(\frac{-c}{p}\right)}{p} \right)^{-1}
+ O\left(\frac{x}{\log^\gamma x}\right). 
\end{align*}

\begin{remark}

In the latter case we have assumed that $4|m$ rather than just $m$ even. This is because if we have a congruence condition $p \equiv c \pmod m$ where $m=2m'$ for $m'$ odd, then it can be decomposed into the conditions $p \equiv c \pmod {m'}$ and $p \equiv 1 \pmod 2$ (since $(c,m)=1$ by assumption). The second condition can then basically be ignored - all it does is exclude the prime 2, and thus does not affect the asymptotic expression. 
\end{remark}
Summing the expressions for the $f=1$ and $f=2$ cases, we obtain
\begin{align*}
\sum_{f=1,2}\frac{1}{f} \sum_{\substack {p \in \delta_f (x)\\ p \in P}} L(1,\chi_{d})\sqrt{p} \log p
= K_P x + O\left(\frac{x}{\log^\gamma x}\right)
\end{align*}
for some constant $K_P$. 
\end{proof}
We apply the above to~(\ref{cross star}) and then substitute back into~(\ref{cross}) to get 
\begin{align} \label{fo}
\frac{1}{2} \sum_{3 < p \leq x} \frac{H(-4p)}{p}
= \frac{2}{\pi} K_P \frac{\sqrt{x}}{\log x}
+ o\left(\frac{\sqrt{x}}{\log x}\right)
\end{align}
and thus we note that $C_P=\frac{2}{\pi}K_P$.

\subsection{The value of $K_P$}
For odd $m$ we have:
\begin{align*}
 K_P=& \frac{\zeta(2)}{\varphi(m)} \left(\prod_{p|m} \left( 1-\frac{1}{p^2}\right) \right) \sum_{k|m}  \left(\frac{-c}{k}\right) \frac{1}{k} \prod_{p|k, p \nmid m/k} \left( 1-\frac{\left(\frac{-c}{p}\right)}{p} \right)^{-1},
\end{align*}
and for even $m$:
\begin{align*}
 K_P=& \frac{\zeta(2)}{\varphi (m)} \left(\prod_{p|m} \left(1-\frac{1}{p^2}\right) \right) \sum_{\text{odd } k|m} \left(\frac{-c}{k}\right) \frac{1}{k}\prod_{p|k, p \nmid m/k} \left( 1-\frac{\left(\frac{-c}{p}\right)}{p} \right)^{-1} \\
+& \frac{I(c)\zeta(2)}{2 \varphi(m)} \left(\prod_{p|m/4} \left(1-\frac{1}{p^2}\right) \right) \sum_{k|m/4} \left(\frac{-c}{k}\right) \frac{1}{k}\prod_{p|k, p \nmid m/4k} \left( 1-\frac{\left(\frac{-c}{p}\right)}{p} \right)^{-1} 
\end{align*}
where $I(c)=1$ if $c$ is congruent to $3 \pmod 4$ and is zero otherwise. 

\section{Applications} \label{s5}
\subsection{Imaginary quadratic fields}
Recall that, by way of example, we mentioned the imaginary quadratic field $\Q(\sqrt{-3})$ in the introduction. The set of rational primes that split in this field is $P=\{p \mid p \equiv 1 \pmod 3 \}$, which gives us the asymptotic 
\begin{align*}
&\frac{1}{4AB} \sum_{|a| \leq A} \sum_{|b| \leq B} \pi_0 \left( \Q\left(\sqrt{-3}\right), E_{a,b}, x \right)
\sim \frac{\pi}{9} \frac{\sqrt{x}}{\log x} 
\end{align*}
for $A,B > x^ {1/2+\epsilon}, AB > x^{3/2+\epsilon}$.

\begin{remark}
 For $P=\{p \text{ prime} \mid p \equiv 2\pmod 3\}$ we have $C_P=2\pi/9$, 
 and thus the occurrence of supersingular primes congruent to $2 \pmod 3$ is significantly
 greater, whereas our heuristic does not distinguish between the constants for the $1\pmod 3$ case and the $2 \pmod 3$ case.
\end{remark}

In general, we have the following constant for an imaginary quadratic field $L=\Q(\sqrt{-q})$, where $q$ is an odd prime: if $q \equiv 3 \pmod 4$, then $C_L = \frac{\pi}{3} \cdot \frac{1}{2} \left(\frac{q-1}{q}\right)$, whereas if $q\equiv 1 \pmod 4$, then $C_L = \frac{\pi}{3} \cdot \frac{1}{2} \left(\frac{q-\frac{1}{4}}{q}\right)$.
So for any imaginary quadratic field there is a bias against the occurrence of supersingular primes that split in the field.

\subsection{Real quadratic fields}

For $L=\Q (\sqrt{q})$, if $q \equiv 3 \pmod 4$, we have $C_L = \frac{\pi}{3} \cdot \frac{1}{2} \left(\frac{q + \frac{1}{4}}{q} \right)$, whereas if $q\equiv 1 \pmod 4$, then $C_L = \frac{\pi}{3} \cdot \frac{1}{2} \left(\frac{q+1}{q} \right)$.
\begin{remark}
 Here we observe a bias toward the occurrence of supersingular split primes. Note that these biases will still be present in the refined averaging of the next section.
\end{remark}
\subsection{Cyclotomic fields}
We could also consider a cyclotomic field such as $\Q(\zeta_{15})$, in which case the set of split primes is $P = \{p \text{ prime} \mid p \equiv 1 \pmod {15}\}$, and so
 \begin{align*}
K_{P} = &\frac{1}{8} \zeta(2) \left(1- \frac{1}{9} \right) \left(1- \frac{1}{25} \right) \left[ 1- \frac{1}{4} + \frac{1}{4} - \frac{1}{16} \right]
= \zeta(2) \frac{1}{10}.
\end{align*}
Thus for $A,B > x^{1+\epsilon}, AB > x^{3/2+\epsilon}$, we have
 \begin{align*}
&\frac{1}{4AB} \sum_{|a| \leq A} \sum_{|b| \leq B} \pi_0 (P,E_{a,b},x)
\sim \frac{\pi}{30} \frac{\sqrt{x}}{\log x}.
\end{align*}

\begin{remark}
The sum of the $C_P$-coefficients for all the various mod 15 congruence relations is $\pi/3$ and thus the mean value is $\pi/24$, which is larger than the coefficient in the asymptotic above. So again we have a bias against the occurrence of supersingular primes that split totally in the number field.
\end{remark}
\begin{remark}
 The bias in the distribution of supersingular primes, as seen in these examples, can be traced back to the $L$-functions of section~\ref{s4}. Consider the inner sum, for $f=1$ say, of the right-hand side of equation~(\ref{cross}) and express this using Euler products to get
\begin{align*}
 & \sum_{\substack{p \in \delta_1 (x)\\ p \in P}} \frac{1}{\sqrt{p}} \cdot \prod_{\text{prime }q} \left( 1-\frac{\left(\frac{-p}{q}\right)}{q} \right)^{-1}.
\end{align*}
Define $Q$ to be the set of (rational) primes congruent to $2 \pmod 3$ and $R$ the set of primes congruent to $1 \pmod 3$. If $P=Q$, the second factor in the Euler product would be $\left(1-\frac{1}{3}\right)^{-1} = 3/2$, whereas if $P=R$, that same factor would be $\left(1+\frac{1}{3}\right)^{-1} = 3/4$. This suggests that choosing a set of primes such as $Q$ leads to a larger constant in the averaging expression. Furthermore, we note that the ratio of the two factors is greater (and thus the bias more pronounced) when $q$ is a smaller prime.
\end{remark}

\section{Refining the sum} \label{s6}

The averaging that we carried out included more than one elliptic curve from each isomorphism class. This repetition can be avoided using a construct of Fouvry-Murty. Define a `set of minimality' $\mathcal{M} = \{(a,b) \in \Z^2 \mid p^2|a \Rightarrow p^3 \nmid b\}$. Then a straightforward extension of a theorem from~\cite{FM96},  under the conditions $A,B > x^{1+\epsilon}$ and $AB > x^{2+\epsilon} \cdot \min (A^{1/4},B^{1/6})$, gives us
\begin{align*}
\sum_{|a| \leq A} \sum_{\substack{|b| \leq B \\ (a,b) \in \mathcal{M}}} \pi_0(P,E_{a,b},x)
=& \frac{2AB}{\zeta(10)} \sum_{\substack{p \leq x \\ p \in P}} \frac{H(-4p)}{p} \Bigl( 1 
+ O\left(p^{-1}\right)
+ O\left(\log^{-4}x\right) \Bigr) \\
&+ O\left(AB \log x\right). 
\end{align*}
We note that the double sum is now no longer over multiple representatives of any given isomorphism class.
Bounding the error terms with~(\ref{wo}), and then applying~(\ref{fo}) we obtain
\begin{theorem}
For $A,B > x^ {1+\epsilon}$ and $AB > x^{2+\epsilon} \cdot \min(A^{1/4},B^{1/6})$,
\begin{align*}
\sum_{|a| \leq A} \sum_{\substack{|b| \leq B \\ (a,b) \in \mathcal{M}}} \pi_0 (P,E_{a,b},x) \sim \frac{4AB}{\zeta(10)}C_P  \frac{\sqrt{x}}{\log x} 
\end{align*}
as $x \rightarrow \infty$.
\end{theorem}

\subsection*{Acknowledgements}
The author would like to thank his advisor Dinakar Ramakrishnan for his guidance and constructive discussions, as well as Chantal David and Noam Elkies for their helpful suggestions on an earlier version of this paper.

This work was done while the author was a graduate student at Caltech and represents part of his Ph.D. dissertation.

\end{document}